\documentclass[12pt,reqno]{article}
\usepackage{amsmath,amsthm,amssymb}
\usepackage{graphicx,color}

\numberwithin{equation}{section}

\def\R{\mathbb{R}}
\def\C{\mathbb{C}}
\def\E{\mathbb{E}}

\def\Z{\mathbb{Z}}
\def\P{\mathbb{P}}
\def\Im{\operatorname{Im}}
\def\Re{\operatorname{Re}}

\newcommand{\zn}{Z_{\beta, N}}
\newcommand{\psin}{\Psi_{N}(\beta)}
\newcommand{\di}{\mathrm{d}}
\newcommand{\xib}{\xi(\beta)}
\newcommand{\hb}{\widetilde{h}_\beta}
\newcommand{\ub}{u_\beta}

\newtheorem{theorem}{Theorem}[section]
\newtheorem{proposition}[theorem]{Proposition}
\newtheorem{lemma}[theorem]{Lemma}
\newtheorem{claim}[theorem]{Claim}
\newtheorem{remark}{Remark}[section]
\newtheorem{corollary}[theorem]{Corollary}
\newtheorem{conj}[theorem]{Conjecture}

\begin{document}
\title{The Curie--Weiss model with complex temperature: phase transitions}
\author{
Mira Shamis\textsuperscript{1},  Ofer Zeitouni\textsuperscript{2}}
\footnotetext[1]{Department of Mathematics, 
Weizmann Institute of Science, Rehovot 7610001, Israel and School of Mathematical Sciences,
Queen Mary University of London, 
Mile End Road, London E1 4NS, England. E-mail:
m.shamis@qmul.ac.uk. Supported in part by ISF grant 147/15.}  \footnotetext[2]{Department of Mathematics, 
Weizmann Institute of Science, Rehovot 7610001, Israel. E-mail:
ofer.zeitouni@weizmann.ac.il. Supported in part by ISF grant 147/15. This project has received funding from the European Research Council (ERC) under the European Union's Horizon 2020 research and innovation programme (grant agreement No. 692452).}
\maketitle

\begin{abstract}
We study the partition function of the Curie--Weiss model
with complex temperature, and partially describe its phase transitions. 
As a consequence, we obtain information on the locations of zeros of the
partition function (the Fisher zeros).
\end{abstract}

\section{Introduction} 
An important component of large deviations theory is Varadhan's lemma, 
which states that if a sequence of probability
measures $\mu_N$ satisfies 
the large deviations principle in a (Polish) space $\mathcal{X}$
with speed $N$ and rate function $I$,
then for any bounded continuous function $f:\mathcal{X}\to \R$,
\begin{equation}
  \label{eq-1}
  \lim_{N\to\infty} \frac1N \log \int e^{N f(x)}\mu_N(dx) =
\sup_{x\in \mathcal{X}} (f(x)-I(x))\,.
\end{equation}
See \cite{DeZe} for a precise statement, relaxed 
assumptions, and applications.

In many applications, considering real-valued $f$ is too restrictive, and one may be 
interested in relaxing it to allow for complex-valued $f$. Statistical
mechanics provides for a rich class of examples; we mention in particular
the Yang--Lee theory \cite{LY}, where the complex 
perturbation is in form of a magnetic 
field, or the quantum spin chain models \cite{Maillet+}, 
where quantitites of interest such as emptiness formation can be formulated 
as  exponential asymptotics of the type \eqref{eq-1} with complex integrand.
Note that in such examples, because $f$ is multiplied by $N$, relatively
small changes in phase may lead to sign changes  of the integrand in
\eqref{eq-1} and therefore to cancelations.

It seems maybe naive  at this point to hope for a general theory, which would
consist of an analogue of \eqref{eq-1}. 
Our goal
in this paper is more modest: we consider one simple example, the
Curie--Weiss model with complex temperature, and partially develop
the asymptotic theory concerning  its partition function. While we are not able
to give a complete description of the associated phase diagram, we will 
be able to show that the phase diagram is not trivial.
 As a consequence
of our analysis, we will also obtain information on the (complex) zeros of the
partition function, which are called the Fisher zeros; see \cite{Fisher} for a discussion of the relations between these zeros and various critical exponents.

The Curie--Weiss model at complex temperature 
was also recently considered by Krasnytska et al.\ \cite{KBHK}; the focus of their work is on the Fisher
zeros in the vicinity of the critical point. We further comment on their results at the end of this introduction.

\smallskip
We begin by introducing the
Curie--Weiss model that we will consider.
Let $\sigma = (\sigma_1, \dots, \sigma_N)\in \{ -1, +1\}^N$. Define
the \textit{Hamiltonian}
\begin{equation}\label{eq:cw}
H_N(\sigma) = -\frac{1}{2N}\sum_{i,j = 1}^N \sigma_i\sigma_j = 
-\frac{N}{2} (m_N(\sigma))^2 ,
\end{equation}
where the \textit{magnetization} is
$m_N(\sigma) = \frac{1}{N}\sum_{i = 1}^N \sigma_i$. For $\beta \in \C$,
let $\zn$ denote the \textit{partition function}, i.e.
\begin{equation}\label{eq:z_n}
\zn = \frac{1}{2^N}\sum_{\sigma\in \{ -1, +1\}^N} \exp(-\beta H_N(\sigma))=
\int\cdots\int \exp(-\beta H_N(\sigma))\prod_{i=1}^N \mu(d\sigma_i), 
\end{equation}
where $\mu(d\sigma)=\frac12 (\delta_1+\delta_{-1})$.

When $\beta$ is real, it is an easy exercise to apply Varadhan's
lemma \eqref{eq-1} and Cramer's theorem concerning the large deviations
of $m_N$ in order to conclude that 
\[
F_\beta = \lim_{N\to\infty}\frac{1}{N}\log |\zn| =
\begin{cases}
  0, & \beta \in (-\infty, 1]\\
  >0, &\beta \in (1,\infty),
\end{cases}
\]
where $-F_\beta/\beta $ is the \textit{free energy}.
More refined analysis (see e.g. \cite{EN})
yields that for $\beta\in \R\setminus\{1\}$,
\begin{equation}\label{eq:zn-real}
\zn = A_\beta e^{NF_\beta}(1 + o(1)),
\end{equation}
where $A_\beta > 0$ is some constant that depends only on $\beta$;
this  is due to 
the Gaussian nature of the fluctuations of $\sqrt{N} (m_N-m^*(\beta))$,
where $m^*(\beta)$ is the asymptotic magnetization, under the
measure $\exp(-\beta H_N(\sigma))\prod_{i=1}^N \mu(d\sigma_i)/\zn$.
Also, $m^*(\beta)=0$ for $\beta \leq 1$.
\begin{remark}
Here and throughout the paper, the notation $O(\cdot)$ and $o(\cdot)$ 
is used for asymptotics as $N\to\infty$, for fixed $\beta\in\C$. That is, 
$a_N=O(b_N)$ if $\limsup_{N\to\infty} |a_N|/|b_N|<\infty$ and
$a_N=o(b_N)$ if the last $\limsup$ equals $0$. When we want to emphasize
dependence on other parameters, we use the notation $o_{\epsilon,R}$, etc.
\end{remark}

When $\beta=(1+\epsilon+iR)\in \C$, 
one expects to similarly have a separation between
a region where $F_\beta=0$ and $F_\beta\neq 0$. In particular, 
one predicts the existence of a critical curve $\mathcal{C}$  in the
complex plane, passing through $1$, that divides the complex plane
into a region where $F_\beta=0$ and its complement where $F_\beta\neq 0$.

For symmetry reasons, it is enough to consider $R\geq 0$.
Our first result describes a region where $F_\beta$ vanishes.
\begin{theorem}\label{th:zero_part} 
  There exist constants $c,c',\epsilon_0>0$ so that, 
  with $\beta = 1 + \epsilon + iR$, 
  if either   
  $0< \epsilon\leq \epsilon_0$ and
  $c \sqrt{\epsilon}\leq R\leq \frac{c'}{\sqrt{\epsilon}}$ 
  or $\epsilon<0$, 
  then
\begin{equation*}
  \zn = \sqrt{\frac{\beta}{\beta - \beta^2}}(1+o_{\epsilon,R}(1))
  ,\ \quad \lim_{N\to\infty} \frac{1}{N}\log\left|\zn\right| = 0.
\end{equation*}
\end{theorem}

\begin{remark}
  One can make the constants $c,c',\epsilon_0$ explicit. Our proof
  gives $\epsilon_0=1/9,c=\sqrt{20}, c'=\pi/\sqrt{32}$, but these are certainly
  not optimal constants.
\end{remark}
\begin{remark}
  It is possible to also treat the case of $\epsilon=0$, where
  one may observe a transition as function of $R$: for $R=0$, it is 
  standard, see \cite[Theorem 2]{ML}, that $\zn$ is asymptotic to a constant
  multiple of $N^{1/4}$, while a local analysis near 
  the saddle point $0$ reveals that if $R>0$ is small
  then $\zn$ is asymptotic
  to an ($R$-dependent) constant, see Theorem 
  \ref{th:pacman} below.
\end{remark}

Our next result shows that along a particular curve that is asymptotic 
to $1+i\infty$ and to $\infty +\pi i$, indeed $F_\beta>0$. 
\begin{theorem} \label{th:positive_part} For $\beta = 1 +\epsilon + iR$ on the curve
\begin{equation}\label{eq:curve}1 + \epsilon = \frac{R}{2\pi}\log\left(\frac{1 + \frac{\pi}{R}}{1 - \frac{\pi}{R}}\right),\, \quad \pi < R < \infty,
\end{equation}
we have, for some constant $\widetilde{A}_\beta$, that
\[
|\zn| = Z_{\Re \beta, N} \widetilde{A}_\beta (1 + o(1)),\, \quad \lim_{N\to\infty} \frac{1}{N}\log|\zn| > 0.
\]
\end{theorem}
\begin{remark} The curve in Theorem \ref{th:positive_part} is asymptotic
  to $c''/\sqrt{\epsilon}$ as $\epsilon\to 0$; compare with
  Theorem \ref{th:zero_part}, noting that $c''\neq c'$.
\end{remark}

In a neighborhood of $\beta=1$, we actually can give a complete description 
of the transition away from $F_\beta=0$.
Define the even function 
\[ h_\beta(u) = \frac{u^2}{2\beta} - \log\cosh u~, \quad u \in\C\setminus ((-i\infty, -i\pi/2]\cap[i\pi/2, +i\infty))~.\] 
With this definition we will see in  Proposition \ref{prop:int-rep-z} that
\begin{equation}\label{eq:zn-hb}
\zn = \sqrt{\frac{N}{2\pi\beta}}\int_{-\infty}^\infty e^{-Nh_\beta(u)}\di u.
\end{equation}
In Claim \ref{cl:zeros-h} below
we show that for some $c>0$ small and  
$0 < |\beta - 1| \leq c$,
$h'_\beta(u)$ has three zeros in a neighborhood of $0$: $0, \pm u_\beta$.

\begin{theorem}\label{th:pacman} There exists $c' \leq c$ such that for $0 < |\beta - 1| \leq c'$
\begin{enumerate}\label{en:pacman}
\item\label{pacman-th:eq1} $\zn = \frac{1}{\sqrt{1 - \beta}}\left( 1 + O\left(\frac{1}{N}\right)\right)$ when $\Re\beta \leq 1$,
\item\label{pacman-th:eq2} $\zn = \frac{1}{\sqrt{1 - \beta}}\left( 1 + O\left(\frac{1}{N}\right)\right) + 2\sqrt{\frac{\beta}{\beta - \beta^2 + u_\beta^2}}e^{-Nh_\beta(u_\beta)}\left( 1 + O\left(\frac{1}{N}\right)\right)$ when $\Re\beta \geq 1$,
\end{enumerate}
and for any $\delta > 0$ the implicit constants are uniform in $\beta$ with $\delta \leq |\beta - 1| \leq c'$. 
\end{theorem}
See Figure~\ref{fig:pacman} for a schematic illustration of our theorems.
We remark that on the line $\Re\beta = 1$ we will see (as a consequence of Claim
\ref{cl:crit-curve} below) that
$\Re h_\beta(\ub) > 0$ (except for $\beta = 1$). 
In particular,
the two statements in  Theorem \ref{th:pacman}
coincide  on that line.

%
\begin{figure}\label{fig:pacman}
\begin{center}
\includegraphics[scale=.63]{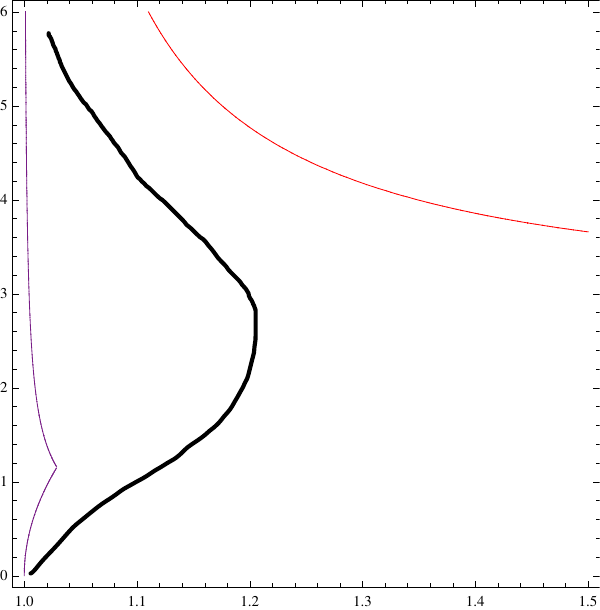}\qquad
\includegraphics[scale=1]{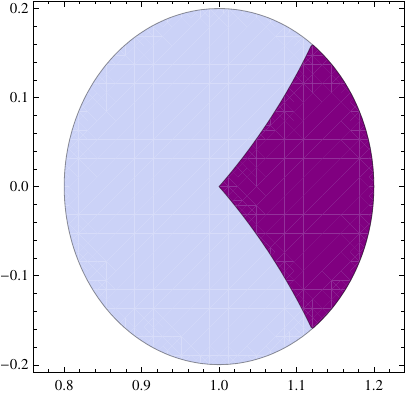}
\end{center}
\caption{Schematic illustration of our results.
  {\bf Left}: $F_\beta = 0$ to the left of the purple curve on the left (Theorem~\ref{th:zero_part});
$F_\beta < 0$ on  the red curve on the right (Theorem~\ref{th:positive_part}). We conjecture
that the two phases are separated by a curve similar to the one schematically depicted in black. The three curves are asymptotic at infinity to the line $\Re \beta = 1$.
{\bf Right}: The vicinity of the critical point $\beta=1$: here $F_\beta >0$ in the purple region on the right, and $F_\beta = 0$ in the blue region on the left.}
\end{figure}

In Theorem \ref{th:pacman}, an important role is played by 
those $\beta$ with $\Re h_\beta(\pm u_\beta)=0$. These are characterized 
by the following claim.
\begin{claim}\label{cl:crit-curve} There exist $c,C>0$ and 
  a smooth function $\epsilon\mapsto R_0(\epsilon)$ on $[-c,c]$
  such that $|R_0(\epsilon) - \epsilon| \leq C\epsilon^2$ and 
  the following holds for $\beta = 1 + \epsilon + iR$:
\begin{itemize}
\item If $|R| < |R_0(\epsilon)|$, then $\Re h_\beta(\pm u_\beta) < 0$,
\item If $|R| = |R_0(\epsilon)|$, then $\Re h_\beta(\pm u_\beta) = 0$,
\item If $|R| > |R_0(\epsilon)|$, then $\Re h_\beta(\pm u_\beta) > 0$.
\end{itemize}
\end{claim}
For $c$ as above, define the \textit{critical curve} 
$\Gamma = \{\beta = 1 + \epsilon + i R\, |\,\, 0 \leq\epsilon \leq c,\,\, 
R = \pm R_0(\epsilon) \}$.
Theorem \ref{th:pacman}
allows us to describe the location of zeros of $\zn$ (the Fisher zeros), and show
that in a neighborhood of $\beta=1$,
they are close to the critical curve $\Gamma$. Define 
\[
\psin = \frac{1}{\sqrt{1 - \beta}} + 2\sqrt{\frac{\beta}{\beta - \beta^2 + u_\beta^2}}e^{-Nh_\beta(u_\beta)},\,\, \Re\beta \geq 1.
\]
The zeros of $\psin$ 
near $\beta = 1$ lie near the critical curve $\Gamma$; we will show that
the zeros of $\zn$ are close to those zeros.
\begin{corollary}\label{cor:zeros-of-zn} For any $\delta > 0$ the following holds for $N \geq N_0(\delta)$. The zeros of $\zn$ in $\delta < |\beta - 1| < c'$ lie in $\Re\beta > 1$, and for any zero $\beta$ of $\zn$ there exists a unique zero $\beta'$ of $\psin$ such that $|\beta - \beta'| < \frac{C_\delta}{N^2}$. Vice versa, for any zero $\beta'$ of $\psin$ with $\delta < |\beta' - 1| < c'$ there exists a unique zero $\beta$ of $\zn$ with $|\beta - \beta'| < \frac{C_\delta}{N^2}$. In particular, the zeros of $\zn$ lie within $C_\delta N^{-1}$ from the critical curve.
\end{corollary}

We also obtain information on the empirical measure of zeros of
$Z_N$, in a neighborhood of the critical point $\beta=1$. For $c'>0$ small,
introduce the
scaled zero-counting measure
\begin{equation}\label{eq:zero-count-mes}
\mu_N = \frac{1}{N} \sum_{\beta:\, |\beta - 1|\leq c',\,\, \zn = 0}\delta_\beta.
\end{equation}
Define a positive measure $\mu$ on $\C$,
supported on $\Gamma_{c'}:=\Gamma\cap\{|\beta - 1|\leq c' \}$ as follows: 
For a segment $I$ of $\Gamma$ connecting $a_j=1+\epsilon_j+iR_0(\epsilon_j)
\in \Gamma_{c'}$,
$j=1,2$, with $\epsilon_2>\epsilon_1\geq 0$, 
set
\begin{equation}\label{eq:limit-mes}
  \mu(I) = \frac{1}{2\pi}(\Im h_{a_2}(u_{a_2}) - \Im h_{a_1}(u_{a_1})),
\end{equation}
and extend $\mu$ by symmetry to the lower half plane.
One checks that $\mu$ is a finite positive measure on $\C$.

\begin{corollary}\label{cor:emp-mes} $\mu_N \underset{N\to\infty}{\longrightarrow}
  \mu$ in the weak topology for 
  positive measures on $\C$.
\end{corollary}

We remark that the Fisher zeros in the vicinity of the
critical point $\beta=1$ were recently studied by
Krasnytska et al.\ \cite{KBHK}. One of their
results, derived on the level of rigor customary in the physics literature, describes the asymptotics of the
$j$-th zero in the limiting regime at which $N \to \infty$ and then $j \to \infty$ (see op.\ cit., eq. (28)).
As a consequence, they find that the zeros pinch
the real axis at angle $\pi/4$, which is consistent
with our results. See also \cite{GPS} for an earlier reference.

\smallskip
The results above do not completely characterize the phase diagram of
the Curie--Weiss model. In Section 
\ref{sec-open}, we discuss this point and present a conjecture
for the critical curve separating the region where the free energy 
vanishes asymptotically from that where it is strictly positive.

\section{Integral representation and preliminaries}
The proofs of all of the theorems are based on the saddle-point analysis of the following integral representation.
\begin{proposition}\label{prop:int-rep-z} If $\Re \beta > 0$, then
\begin{equation}\label{eq:int-rep-z}
\zn = 
\left(\frac{\beta N}{2\pi}\right)^{1/2} \int_{-\infty}^{\infty} \exp(-Nf_\beta(u)) \di u =
\left(\frac{N}{2\pi\beta}\right)^{1/2} \int_{-\infty}^{\infty} \exp(-Nh_\beta(u)) \di u 
,
\end{equation}
where 
\begin{equation}\label{eq:f-beta}
f_\beta(u) = \frac{\beta u^2}{2} - \log(\cosh(\beta u)),
\end{equation}
and
\begin{equation}\label{eq:h-beta}
h_\beta(u)=\frac{u^2}{2\beta} - \log \cosh u,
\end{equation}
and the branch of the square root is choosen so that $\sqrt 1 = 1$.
\end{proposition}
\begin{proof}Let $X_1, X_2, \dots, X_N$ be independent identically distributed Bernoulli random variables: $\P\{X_j = 1\} = \P\{X_j = -1\} = \frac{1}{2}$,\ and let $\beta\in\C$. Then,
  using $\E$ to denote expectation with respect to these random variables,
  we have
\begin{equation*}
  \begin{split}
\zn &=  
\E \exp\left(\frac{\beta N}{2}\left( \frac{1}{N}\sum_{j=1}^N X_j\right)^2
\right)=
 \E\int \exp\left( -\frac{u^2}{2} + u\sqrt{\beta N}\frac{\sum X_j}{N}\right)\frac{\di u}{\sqrt{2\pi}}   \\ &=  
 \sqrt{\frac{\beta N}{2\pi}}\int \di\tilde{u} \exp\left( -\frac{\beta N}{2}\tilde{u}^2\right) \E \exp\left(\beta\tilde{u}\sum_{j=1}^N X_j\right),
\end{split}
\end{equation*}
where the second equality uses the
Hubbard-Stratonovich transformation and the last uses the
change of variables $\tilde{u} = \frac{u}{\sqrt{\beta N}}$. 
Since for any $a$ we have $\E\exp(aX) = \cosh a$ 
and using the assumption that $X_j$ are i.i.d random variables, we obtain
\[
\E \exp\left(\beta\tilde{u}\sum_{j=1}^N X_j\right) = \prod_{j=1}^N \E \exp\left(\beta\tilde{u} X_j\right) = \left( \E \exp\left(\beta\tilde{u} X_j\right) \right)^N = \left(\cosh(\beta\tilde{u})\right)^N.
\]
Combining the last two displays gives 
\begin{equation*}
\zn=
\sqrt{\frac{\beta N}{2\pi}}
\int \di\tilde{u} \exp\left( -\frac{\beta N}{2}\tilde{u}^2 + N\log\cosh(\beta\tilde{u})\right)=
\sqrt{\frac{\beta N}{2\pi}}
\int \exp(-N f_\beta(u))du
.
\end{equation*}
\end{proof}

\section{Proof of Theorem \ref{th:zero_part}}
The proof of Theorem \ref{th:zero_part} for $\epsilon< 0$ follows
from known asymptotics for the Curie--Weiss model. Indeed,
for such $\epsilon$ and with $R=0$ and $\sigma_\epsilon^2=1/(1-\epsilon)$,
we have by \cite[Theorem 2]{ML} that 
$Z_{1+\epsilon,N}\asymp \sigma_\epsilon $ and, under the measure
$e^{-(1+\epsilon) H_{N}(\sigma)}/Z_{1+\epsilon,N}$, we have by 
\cite{EN} that $\sqrt{N}m_N$ 
converges in
 distribution to a centered Gaussian random variable
  of variance $\sigma_\epsilon^2:=-1/\epsilon$. We then obtain
  that with $\epsilon <0$, 
  $$Z_{\beta,N}  \asymp \frac{1}{\sqrt{2\pi}}\int 
  e^{-(1-\epsilon)u^2/2-iRu^2/2} du\,,$$
  which gives the claim.


The proof in case $\epsilon>0$ follows a saddle-point 
analysis of the integral representation from Proposition \ref{prop:int-rep-z}. 
Throughout, $C_i$ denote constants that may depend on $\epsilon$ and $R$
but not on anything else.
The following preliminary claims play an important role in the analysis. 
\begin{claim}\label{claim:w1} For any $\epsilon\leq\frac{1}{9}$, 
  $u\leq \sqrt{8\epsilon}$, $\sqrt{24\epsilon}\leq R\leq \frac{\pi}{\sqrt
    {128
    \epsilon}}$,
\[
\Re f_\beta(u) = (1 + \epsilon)\frac{u^2}{2} - \frac{1}{2}\log (\cosh^2(u(1 +\epsilon)) - \sin^2(uR)) \geq \epsilon\frac{u^2}{2}.
\]
\end{claim}
\begin{proof} By Taylor expansion,
\[
\cosh t \leq 1 + \frac{t^2}{2} + \cosh t \frac{t^4}{24} \leq 1 + \frac{t^2}{2} + \frac{t^4}{12}
\]
where the last inequality used that $t\leq 3/2$ and therefore
$\cosh t\leq 2.4$. Using again $t<3/2$ we obtain
\[
\cosh^2 t \leq 1 + t^2 + \frac{3}{4} t^4.
\]
From the assumptions we get that $(1+\epsilon)u\leq
(1+\epsilon)\sqrt{8\epsilon}<3/2$. Therefore, using again
$\epsilon<1/9$,
\[
\cosh^2 ((1 +\epsilon)u)  \leq 1 + u^2 +\epsilon u^2\left(2 +\epsilon + 
\frac{3 u^2 (1+\epsilon)^4}{4\epsilon}\right)\leq
1+ u^2 + 12\epsilon u^2.
\]
For $Ru\leq\frac{\pi}{4}$ we have $\sin^2 (Ru)\geq {R^2 u^2}/{2}$, hence
\[
\cosh^2 ((1 +\epsilon)u) - \sin^2 (Ru) \leq 1 + u^2 + 12\epsilon u^2 - \frac{R^2 u^2}{2}. 
\]
Since $R\geq\sqrt{24\epsilon}$, we have
${R^2 u^2}/{2} \geq 12 \epsilon u^2$ and therefore
\[
\cosh^2 ((1 +\epsilon)u) - \sin^2 (Ru) \leq 1 + u^2 \leq e^{u^2},
\] 
and therefore
\[
\Re f_\beta(u) = (1 + \epsilon)\frac{u^2}{2} - \frac{1}{2}\log (\cosh^2(u(1 +\epsilon)) - \sin^2(uR)) 
\geq \epsilon \frac{u^2}{2}.
\]
\end{proof}
The next claim handles larger values of the argument $u$.
\begin{claim}\label{lemma:cosh_estimate}
  Let $\epsilon<1/9$.
For any $t\geq (1+\epsilon)
\sqrt{8\epsilon}$,
\begin{equation}
  \label{eq-prelim2a}
\cosh t \leq \exp\left((1-\epsilon)\frac{t^2}{2}\right).
\end{equation}
In particular, for $u\geq \sqrt{8\epsilon}$, 
\begin{equation}
  \label{eq-prelim2}
  \Re f_\beta(u)\geq (1+\epsilon)\frac{u^2}{2}-\log (\cosh(u(1+\epsilon))\geq 
  \frac{\epsilon^2 u^2}{2}.
\end{equation}
\end{claim}
\begin{proof}
By Taylor expansion we have
\[
\cosh t = \sum_{k=0}^\infty \frac{t^{2k}}{(2k)!} = 1+ \frac{t^2}{2}(1-\epsilon) +\left[ \epsilon\frac{t^2}{2} + \frac{t^4}{24}\right] + \sum_{k\geq 3}^\infty \frac{t^{2k}}{(2k)!},
\]
and
\[
\exp\left((1 - \epsilon)\frac{t^2}{2}\right) = 1+ \frac{t^2}{2}(1-\epsilon)
+ \frac{t^4}{8}(1-\epsilon)^2 +  \sum_{k\geq 3}^\infty \frac{t^{2k}(1 - \epsilon)^k}{2^kk!}.
\]
For $k \geq 3$ we have, if $\epsilon\leq 1/2$,
\[
\frac{(2k)!}{2^kk!} = \frac{(k + 1) \cdots (2k)}{2^k} \geq 2^k \geq \frac{1}{(1 - \epsilon)^k}.
\]
Therefore
\[
\sum_{k\geq 3}^\infty \frac{t^{2k}}{(2k)!} \leq \sum_{k\geq 3}^\infty \frac{t^{2k}(1 - \epsilon)^k}{2^kk!},
\]
and, since $\epsilon\leq \frac{1}{9}$ and $t\geq (1+\epsilon)\sqrt{8\epsilon}$,
\[
\frac{t^4}{8}(1 - \epsilon)^2 - \frac{t^4}{24} - \frac{\epsilon
t^2}{2} \geq \frac{t^4}{8}\left((1-\epsilon)^2-\frac13-
\frac{1}{2(1+\epsilon)^2}
\right)
\geq 0.
\]
This completes the proof of \eqref{eq-prelim2a}.

To see \eqref{eq-prelim2}, take $t=(1+\epsilon)u$, which satisfies the
assumptions leading to \eqref{eq-prelim2a}. Then,
\begin{equation}\label{eq:cosh_estim}
\cosh((1+\epsilon)u)\leq \exp\left((1-\epsilon)(1+\epsilon)^2\frac{u^2}{2}\right) \leq \exp\left(\frac{(1+\epsilon)(1-\epsilon^2)u^2}{2}\right).
\end{equation}   
Hence, we have, using the monotonicity of the logarithm and 
(\ref{eq:cosh_estim})
\begin{equation*}
(1 + \epsilon)\frac{u^2}{2} - 
\log (\cosh(u(1 +\epsilon)))\geq 
(1 + \epsilon)\frac{u^2}{2} - \frac{(1+\epsilon)(1-\epsilon^2)u^2}{2} \geq 
\frac{\epsilon^2 u^2}{2}.
\end{equation*}
\end{proof}

We continue with the proof of Theorem \ref{th:zero_part}, considering
the regime $\epsilon>0$, and $R$ as in the statement of Claim \ref{claim:w1}.
%
In view of Proposition \ref{prop:int-rep-z}, we write
\begin{equation}\label{eq:split_of_integral}
\int_{-\infty}^\infty e^{-Nf_\beta(u)}\di u = \left[\int_{-\infty}^{-\delta} + \int_{\delta}^\infty + \int_{-\delta}^{\delta}\right] e^{-Nf_\beta(u)}\di u = I_1' + I_1+ I_2 = 2I_1 + I_2,
\end{equation}
where the last equality follows from the symmetry. Note that 
$f_\beta(0)=f'_\beta(0) = 0$. Hence, $u = 0$ is a saddle point. We will show below that the main contribution to the integral comes from a neighborhood of this saddle 
point. We will choose 
$\delta = N^{-2/5}$ so that 
$N\delta^3 \to 0$ and $\delta\sqrt{N}\to \infty$ as $N\to\infty$. 

We begin by estimating $I_1$.
\[
I_{1} = \int_{\delta}^\infty e^{-Nf_\beta(u)}\di u = 
\left[\int_{\delta}^{\sqrt{8\epsilon}} + \int_{\sqrt{8\epsilon}}^\infty \right]  e^{-Nf_\beta(u)}\di u = W_1 + W_2.
\]
Using Claim \ref{claim:w1}, we have
\begin{equation}\label{eq:w1}
  |W_1|\leq \int_{\delta}^{\sqrt{8\epsilon}} e^{-N\Re f_\beta(u)}\di u 
  \leq \int_{\delta}^{\sqrt{8\epsilon}} e^{-N\frac{\epsilon^2
  u^2}{2}}\di u \leq e^{-N\frac{
  \epsilon^2\delta^2}{2}}(\sqrt{8\epsilon} - \delta)\leq e^{-cN^{1/5}},
\end{equation}
for some constant $c>0$.\\
\medskip
To estimate $W_2$, we use \eqref{eq-prelim2a} of
Claim \ref{lemma:cosh_estimate}  and obtain
\begin{equation}\label{eq:w2}
  |W_2| \leq\int_{\sqrt{8\epsilon}}^\infty
  e^{-N\Re f_\beta(u)} \di u\leq \int_{\sqrt{8\epsilon}}^\infty
  e^{-N\epsilon\frac{u^2}{2}}\mathrm{du}\leq e^{-CN},
\end{equation}
where $C=C(\epsilon)>0$ is some constant. Combining (\ref{eq:w1}) and 
(\ref{eq:w2}) we get
\begin{equation}\label{eq:j1} |I_{1}|\leq e^{-CN^{1/5}} + e^{-CN} \leq e^{-\tilde{C}N^{1/5}}.
\end{equation}
We turn to estimating
$I_2$. Denote by 
\[
P_2(u) = (\beta - \beta^2)\frac{u^2}{2} 
\]
the Taylor approximation of $f_\beta(u)$ to second order. Note that, by the assumptions on $R$ 
\[
\Re(\beta - \beta^2) = -\epsilon -\epsilon^2 + R^2 \geq -2\epsilon + R^2 > 0. 
\]
For $|u|< \delta$ for our choice of $\delta$\, we have $O(u^3) = O(N^{-6/5})$. We get
\begin{equation}\label{eq:diff-f-fT}\begin{split}
\int_{-\delta}^\delta e^{-Nf_\beta(u)}\di u &= \int_{-\delta}^\delta e^{-NP_2(u)}e^{-N(f_\beta(u) - P_2(u))}\di u \\ &= \int_{-\delta}^\delta e^{-NP_2(u)}\di u + \int_{-\delta}^\delta e^{-NP_2(u)}\left(e^{-N(f_\beta(u) - P_2(u))} -1\right)\di u.
\end{split}\end{equation}
Since for any $|x| < 1/2$: $|e^x - 1| < 2x$, we obtain
\begin{equation}\label{eq:e-1}\left|e^{-N(f_\beta(u) - P_2(u))} -1\right| \leq C_1N|f_\beta(u) - P_2(u)|\leq C_2 N^{-1/5},
\end{equation}
where the constants depend only on $\beta$.
Combining (\ref{eq:diff-f-fT}) and (\ref{eq:e-1}) we obtain
\begin{equation}\label{eq:diff-f-fT-intdelta}
\left|\int_{-\delta}^\delta e^{-Nf_\beta(u)}\di u - \int_{-\delta}^\delta e^{-NP_2(u)}\di u\right|\leq C_2N^{-1/5}\int_{-\delta}^\delta e^{-N\Re P_2(u)}\di u \leq C_3N^{-1/5}\frac{1}{\sqrt{N}},
\end{equation}
Now we have the following inequality
\begin{equation}\label{eq:diff-fT} \begin{split}
\left| \int_{-\infty}^\infty e^{-NP_2(u)}\di u - \int_{-\delta}^\delta e^{-NP_2(u)}\di u \right|& \leq 2\int_\delta^\infty e^{-N\Re P_2(u)}\di u \leq 2\int_\delta^\infty e^{-N\Re (\beta - \beta^2)\frac{u\delta}{2}} \di u \\ &=  \frac{1}{C_4N^{3/5}}e^{-C_5N^{1/5}}.
\end{split}\end{equation}
Since
\begin{equation*}
\int_{-\infty}^\infty e^{-N P_2(u)}\di u = \sqrt{\frac{2\pi}{N(\beta - \beta^2)}},
\end{equation*}
we obtain combining (\ref{eq:diff-f-fT-intdelta}) and (\ref{eq:diff-fT})
that
\begin{equation*}
\left| I_2 - \sqrt{\frac{2\pi}{N(\beta - \beta^2)}}\right|\leq \frac{C_6}{N^{7/10}}.
\end{equation*}
Using the estimate (\ref{eq:j1}) on $|I_1|$ we obtain
\begin{equation*}
\left| \int_{-\infty}^\infty e^{-Nf_\beta(u)}\mathrm{du} - \sqrt{\frac{2\pi}{N(\beta - \beta^2)}} \right| \leq \frac{C_7}{N^{7/10}}.
\end{equation*}
This concludes the proof of the theorem.
\qed

\section{Proof of Theorem \ref{th:positive_part}}

\begin{proof} Observe that 
\[
f_{1 + \epsilon}(u) = (1 + \epsilon)\frac{u^2}{2} - \log(\cosh((1 + \epsilon)u)),\, \quad f'_{1 + \epsilon} (u)= (1 + \epsilon)\left(u - \tanh((1 + \epsilon)u))\right).
\] 
By the assumption (\ref{eq:curve}) we get
\[
f'_{1 + \epsilon}\left(\pm\frac{\pi}{R}\right) = (1 + \epsilon)\left( \pm\frac{\pi}{R} - \tanh \left( \pm\frac{\pi}{R}\frac{R}{2\pi}\log\left(\frac{1 + \frac{\pi}{R}}{1 - \frac{\pi}{R}}\right)\right) \right) = 0, \, \, \text{and}\, \,  f'_{1 + \epsilon} (0) = 0.
\]
These are the only real zeros of 
$f'_{1 + \epsilon}$ and 
$f_{1 + \epsilon}\left(\pm\frac{\pi}{R}\right) < 0$, in particular the 
minimum of $f_{1+\epsilon}(u)$ over $\R$ is achieved 
at $u=\pm \pi/R$.
%
We claim that the same is true of 
\[
\Re f_{\beta}(u) = (1 + \epsilon)\frac{u^2}{2} - \frac{1}{2}\log (\cosh^2((1 + \epsilon)u) - \sin^2(Ru)).
\]
Indeed, 
by the monotonicity of the logarithm, 
for any $u\in\R$ we get $\Re f_\beta (u) \geq f_{1 + \epsilon} (u)$.
On the other hand,
$\Re f_{\beta} (\pm\frac{\pi}{R}) =f_{1 + \epsilon} (\pm\frac{\pi}{R}) < 0$.
This yields that the minimum of $\Re f_\beta(u)$ over $\R$ is achieved 
at $\pm \pi/R$, as claimed.
We note in passing that $f_\beta'(\pm \pi/R)=0$, i.e. the points $\pm \pi/R$,
which minimize $\Re f_\beta (u)$ over $u\in \R$, are in fact saddle points.

Now we estimate the integral as before. 
Let $\delta=N^{-2/5}$ as in the proof of Theorem \ref{th:zero_part}. Similarly to the proof of Theorem \ref{th:zero_part} we divide the integral into pieces and use the symmetry to obtain
\begin{equation*}
\int_{-\infty}^\infty e^{-Nf_\beta(u)}\di u = \left[ 2\int_0^{\frac{\pi}{R} - \delta} + 2\int_{\frac{\pi}{R} - \delta}^{\frac{\pi}{R} + \delta} + 2\int_{\frac{\pi}{R} + \delta}^\infty\right] e^{-Nf_\beta(u)}\di u = 2I_1 + 2I_2 + 2I_3.
\end{equation*}
Also let 
\[
\int_{-\infty}^\infty e^{-Nf_{1 + \epsilon}(u)}\di u = 
\left[ 2\int_0^{\frac{\pi}{R} - \delta} +
  2\int_{\frac{\pi}{R} - \delta}^{\frac{\pi}{R} + \delta} + 2\int_{\frac{\pi}{R} + \delta}^\infty\right] e^{-Nf_{1 + \epsilon}(u)}\di u = 2\widehat{I}_1 + 2\widehat{I}_2 + 2\widehat{I}_3.
\]
Then, for any $j = 1, 2, 3$ we get $|I_j| \leq \widehat{I}_j$. 

We will show that
\begin{eqnarray}
\widehat{I}_1, \widehat{I}_3&\leq& 
Ce^{-CN^{1/5}}e^{-Nf_{1 + \epsilon}\left(\frac{\pi}{R}\right)},
\label{eq:201116a}\\
  |I_2| \sqrt{|f''_{\beta}(\pi/R)|} (1+o(1))&=&
  \widehat{I}_2\sqrt{f''_{1+\epsilon}(\pi/R)}(1+o(1))
  = \sqrt{\frac{2\pi}{N}}e^{-Nf_{1 + \epsilon}\left(\frac{\pi}{R}\right)},
  \label{eq:201116b}
\end{eqnarray}
and this, together with the fact that
$\Re f_\beta(\pi/R)=f_{1+\epsilon}(\pi/R)<0$,
will prove the theorem and also (\ref{eq:zn-real}).

We start from the estimate of $\widehat{I}_3$. We write
\[
  \widehat{I}_3 = \int_{\frac{\pi}{R} + \delta}^\infty  e^{-Nf_{1 + \epsilon}(u)}\di u = \left[ \int_{\frac{\pi}{R} + \delta}^{3} + 
  \int_{3}^\infty\right] 
  e^{-Nf_{1 + \epsilon}(u)}\di u =: W_1 + W_2.
\]
To estimate $W_2$,
note that since $\cosh(x)\leq e^x$ for $x$ real, we have that for
$u\geq 3$,
  \[f_{1+\epsilon}(u)= (1+\epsilon)u^2/2-\log \cosh((1+\epsilon)u)
  \geq (1+\epsilon)u(u/2-1)\geq (1+\epsilon) u/2. \]
  Thus,  
\begin{equation}\label{eq:w2-th2}
W_2 \leq e^{-3N/2}.
\end{equation}
To estimate $W_1$, note that
for any $3 \geq u > {\pi}/{R} + \delta$ 
the function $u\mapsto f_{1 + \epsilon} (u)$ is increasing 
and therefore
 $f_{1 + \epsilon}\left({\pi}/{R} + \delta\right) \geq f_{1 + \epsilon}\left( 
 {\pi}/{R}\right) + c\delta^2$.
Thus,
\begin{equation}\label{eq:w1-th2}
W_1\leq Ce^{-N\left(f_{1 + \epsilon}\left( \frac{\pi}{R}\right) + c\delta^2\right)} = Ce^{-Nf_{1 + \epsilon}\left(\frac{\pi}{R}\right)}e^{-CN^{1/5}}.
\end{equation}
Combining \eqref{eq:w2-th2} and \eqref{eq:w1-th2} yields \eqref{eq:201116a} for $\widehat{I}_3$. On the other hand, 
since $f_{1 + \epsilon} (u)$ is decreasing for any $0\leq u\leq {\pi}/{R} - \delta$ with the minimum at $\pi/R - \delta$, in the same way we obtain
\begin{equation}\label{eq:i1-th2}
\widehat{I}_1 \leq Ce^{-Nf_{1 + \epsilon}\left(\frac{\pi}{R}\right)}e^{-CN^{1/5}},
\end{equation}
which proves \eqref{eq:201116a} 
for $\widehat{I}_1$.

We turn to the proof of \eqref{eq:201116b}, which follows
a saddle point analysis similar to that done in the proof of Theorem
\ref{th:zero_part}. Recall that $\pi/R$ is a saddle point of $f_\beta$ and let
$\widetilde{P}_{2,\beta}$ denote its second order 
Taylor approximation there, i.e.
\[\widetilde{P}_{2,\beta}(u) := f_{\beta}\left(\frac{\pi}{R}\right) + \left(\beta - \frac{\beta^2}{\cosh^2((1+\epsilon)\frac{\pi}{R})}\right)\frac{(u - \frac{\pi}{R})^2}{2}.\]
As in 
\eqref{eq:diff-f-fT} and \eqref{eq:e-1},
replacing the domain of integration to $[\pi/R-\delta,\pi/R+\delta]$ and
$P_2$ by $\widetilde{P}_{2,\beta}$, we obtain the following analog 
of 
\eqref{eq:diff-f-fT-intdelta}:
\begin{equation}\label{eq:f-f2-on-small-int}
\left|\int_{\frac{\pi}{R} - \delta}^{\frac{\pi}{R} + \delta}e^{-Nf_\beta(u)}\di u - \int_{\frac{\pi}{R} - \delta}^{\frac{\pi}{R} + \delta}e^{-N
  \widetilde{P}_{2,\beta}(u)}\di u
\right| \leq \exp\left(-N\Re f_\beta\left(\frac{\pi}{R}\right)\right) C_1 N^{-1/5}\frac{1}{\sqrt{N}}.
\end{equation}
Similarly to 
\eqref{eq:diff-fT}, we also have
\begin{equation}\label{eq:est-f2}
  \left|\int_{-\infty}^{\infty}e^{-N\widetilde{P}_{2,\beta}(u)}\di u - \int_{\frac{\pi}{R} - \delta}^{\frac{\pi}{R} + \delta}
  e^{-N\widetilde{P}_{2,\beta}(u)}\di u\right| \leq \exp\left(-N\Re f_\beta\left(\frac{\pi}{R}\right)\right) \frac{1}{C_2N^{3/5}}e^{-C_3N^{1/5}}.
\end{equation}
Finally, by Gaussian integration we have
\[
\int_{-\infty}^\infty e^{-N
  \widetilde{P}_{2,\beta}} \di u=
  \exp\left(-Nf_\beta\left(\frac{\pi}{R}\right)\right)
  \sqrt{\frac{2\pi}{Nf_\beta''(\frac{\pi}{R})}}.
  \]
  Combining the last display with 
  \eqref{eq:f-f2-on-small-int} and \eqref{eq:est-f2} gives 
  \eqref{eq:201116b} for $I_2$. The analysis of $\widehat{I}_2$ is identical,
  taking $\beta=1+\epsilon$ in $\widetilde{P}_{2,\beta}$. 
\end{proof}

\section{Proof of Theorem \ref{th:pacman}}\label{sec:pacman}

\subsection{Construction of the 
saddle points for $h_\beta(u)$}\label{subsec:saddles}
We begin with the analysis of the critical points of $h_\beta$.
Let $K$ be a large constant (the choice of $K = 241$ will work). Define the
discs in the complex plane:
\begin{enumerate}
\item $D_0 (\beta) = \{ |u|\leq K|\beta - 1|^{3/2}\}$,
\item $D_+(\beta) = \{ |u - \sqrt{3\left( 1 - \frac{1}{\beta}\right)}|\leq K|\beta - 1|^{3/2}\}$,
\item $D_-(\beta) = \{ |u + \sqrt{3\left( 1 - \frac{1}{\beta}\right)}|\leq K|\beta - 1|^{3/2}\}$,
\end{enumerate}
where the branch of the square-root is chosen so that 
$\Im  \sqrt{3\left(1-\frac1\beta\right)}$ is in the upper 
half plane if $\beta$ 
is in the upper half plane.
For sufficiently small $c$ and for $0 < |\beta - 1| < c$,
these circles are disjoint. 

\begin{claim}\label{cl:zeros-h} For any $\beta$ such that $0 < |\beta - 1| < c$, the function $h'_\beta$ has exactly three zeros in $|u| \leq c_1 = 10\sqrt c$, one in each of the discs: $0\in D_0,\,  u_\beta\in D_+,\, -u_\beta\in D_-$. 
\end{claim}
\begin{proof} Introduce the Taylor approximation of $h_\beta(\cdot)$ up to
  fourth order,
\[
P_4(u) = \left(\frac{1}{\beta} - 1\right)\frac{u^2}{2} + \frac{u^4}{12}.
\]
Then,
\[
P'_4(u) = \left(\frac{1}{\beta} - 1\right)u + \frac{u^3}{3},
\]
and $P'_4(u)$ has exactly three zeros $u = 0,\, \, u_\pm = \pm \sqrt{3\left( 1 - \frac{1}{\beta}\right)}$. 
We will show
that on the boundary of each disc, 
namely on $\partial D_0 \cup\partial D_+ \cup\partial D_-$,
\begin{equation}
  \label{eq-est1}
|P'_4(u) - h'_\beta(u)| < |P'_4(u)|,
\end{equation}
which will show by Rouch\'e's theorem that $h'_\beta(\cdot)$
has a unique zero in each disc.

We check \eqref{eq-est1} 
on $\partial D_+$, the other two case are similar.
Since $h_\beta(u)$ is even,
 all odd coefficients in its Taylor approximation vanish. 
 Next, for any $u\in\C$ with $|u|\leq\frac{\pi}{4}$ 
 we get $|\tanh u|\leq 1$ and  therefore, repeatedly using that 
 $\tanh'(u)=1-\tanh^2(u)$,
\[
|h^{(6)}_\beta(u)| = 
|-16 + 136 \tanh^2 u - 240\tanh^4 u + 120\tanh^6 u| \leq 512,
\]
thence 
\begin{equation}\label{eq:bd-der}
|h'_\beta(u) - P'_4(u)|\leq \frac{512}{5!}|u|^5 \leq 5|u|^5.
\end{equation}
On the boundary $\partial D_+$ we have by a direct computation
\[\begin{split}
|P'_4(u)| &= \frac{1}{3}\left|u + \sqrt{3\left( 1 - \frac{1}{\beta}\right)}\right|\left|u - \sqrt{3\left( 1 - \frac{1}{\beta}\right)}\right||u|\\&\geq \frac{1}{3}K|\beta - 1|^{3/2}\left(\sqrt{12\left|1 - \frac{1}{\beta} \right|}\ - K|\beta - 1|^{3/2}\right)\left(\sqrt{3\left|1 - \frac{1}{\beta} \right|} - K|\beta - 1|^{3/2}\right).
\end{split}
\]
Choosing $c>0$ small enough, we obtain that
if $K|\beta - 1| \leq c$ then
\[
|P'_4(u)| \geq \frac{1}{3}K|\beta - 1|^{3/2} 3 |\beta - 1|^{1/2}|\beta - 1|^{1/2} = K |\beta - 1|^{5/2}.
\]
On the other hand, combining the estimate (\ref{eq:bd-der}) and that we are on the boundary of $D_+$ we obtain
\[
|h'_\beta(u) - P'_4(u)|\leq 5|u|^5 \leq 5\left[ \sqrt{3\left|1 - \frac{1}{\beta} \right|} + K|\beta - 1|^{3/2}\right]^5 \leq 5\cdot 2^5\cdot \frac{3}{2}|\beta - 1|^{5/2} < K|\beta - 1|^{5/2},
\]
since we assumed $K > 240$. Therefore, Rouch\'e's theorem applies and $D_+$ contains exactly one zero of $h'_\beta$.

To see that $h'_\beta$ has no more zeros in $|u| \leq c_1$, note that for such $u$
\[
|h'_\beta(u) - P'_4(u)|\leq 5|u|^5 \leq 5c_1^5,\,\, |P'_4(u)| \geq Cc_1^3,
\]
and we have the needed estimate by adjusting the constant $c_1$ such that $C > 5c_1^2$. Therefore, by an application of Rouch\'e's theorem we obtain the claim.
\end{proof}

\subsection{Proof of Theorem \ref{th:pacman}} Since $h_\beta$ is even, we write $h_\beta(u) = \hb(u^2)$. Note that $\widetilde{h}_1(0) = \widetilde{h}'_1(0) = 0$, while $\widetilde{h}''_1(0) = 1/6$. Also denote by $D(x, \epsilon)$ the disc in the complex plane of radius $\epsilon$ centered at $x$.

We use a change of variables provided by 
a theorem of Levinson,
which reduces $\tilde{h}_\beta$ to a polynomial of degree 2.  Indeed,
by Levinson's theorem \cite{Lev} (see also \cite[Theorem 1]{Martin}, after 
correcting for typos), 
there exist $\rho, c' > 0$ and analytic functions 
\begin{itemize}
\item $V: D(0, \rho)\times D(1, c') \to \C$,\,\,\, $V(0,\beta) = 0$,
\item $U: \{ (v, \beta)\, |\, \beta\in D(1, c'), v\in V(D(0, \rho), \beta)\} \to \C$,
\end{itemize}
such that, for $\beta\in D(1,c')$,
\begin{align}
\label{u-inverse} &U(V(z, \beta), \beta) = z\,\, \text{on}\,\, 
D(0, \rho),\,\, \text{that is},\,\, V\,\, \text{is the inverse of}\,\,  U, \\
\label{v-one-to-one} &V\,\, \text{is one to one on}\,\, D(0, \rho),\,\, \text{and}\,\,  0 < \frac{1}{C}\leq|V'(z,\beta)| \leq C < \infty,\\
\label{h-rep}&\hb(V(z, \beta)) = \frac{z^2}{2} - \xib z,\,\,\,  \text{with}\,\,\, \xib\,\,\, \text{analytic on}\,\,\, D(1, c),\,\,\, \xi(1) = 0,
\end{align}
where, for any function $f=f(z,\beta)$, we
write $f'(z,\beta)=\frac{\partial}{\partial z} f(z,\beta)$.
From (\ref{h-rep}) we obtain 
\begin{equation}\label{eq:partial}
 (\hb(V(z, \beta)))' = z - \xib,
\end{equation}
and therefore, since $|V'(z,\beta)| \neq 0$, one deduces that $\widetilde{h}'_\beta(V(\xib, \beta)) = 0$. In particular, $V(\xib, \beta)$ is a critical point of $\hb$. Since for $0 < |\beta - 1| < c$ the point $\ub^2$ is the 
unique critical point of $\hb$ in a neighborhood of zero, we obtain that
\begin{align}
\label{en:Vxib} &V(\xib, \beta) = \ub^2,\\
\label{en:xib} &\xib = \sqrt{-2 h_\beta(\ub)}.
\end{align}
Using (\ref{eq:partial}) once again and L'H\^opital's Rule, we obtain
\begin{equation}\label{en:V'xib}
V'(\xib, \beta) = \sqrt{\frac{1}{\hb''(V(\xib, \beta))}} = \frac{2\beta\ub}{\sqrt{\beta - \beta^2 + \ub^2}},
\end{equation} 
where the last equality follows since, by a 
direct computation and using (\ref{en:Vxib}), we obtain 
\[
\hb''(V(\xib, \beta)) = \hb''(\ub^2) = \frac{1}{4\ub^2}\left[\frac{1}{\beta} - 1 + \tanh^2\ub \right] = \frac{1}{4\ub^2}\left[\frac{1}{\beta} - 1 + \frac{\ub^2}{\beta^2} \right],
\]
and the last equality follows since $u_\beta$ is the critical point of $h_\beta(u)$ obeying $\tanh\ub = \frac{\ub}{\beta}$. Repeating this computation at $u=0$
we obtain
\begin{equation}\label{en:V'0}
V'(0, \beta) = \frac{2\beta\xib}{\beta - 1}.
\end{equation}

%

\noindent We need to estimate the following integral
\begin{equation}\label{eq:int-division}
\int_{-\infty}^\infty e^{-Nh_\beta(u)}\di u = 2 \int_0^\infty e^{-Nh_\beta(u)}\di u.  
\end{equation}
Let $\nu = |\beta - 1|^{0.1}$ and consider the following change of contour. 
\[
T_1 = [0, \sqrt{V(\nu, \beta)}],\,\, T_2 = [\sqrt{V(\nu, \beta)}, 6^{1/4}\sqrt\nu],\,\, T_3 = [6^{1/4}\sqrt\nu, \infty],
\]
where $V(\nu, \beta) \in\C$ and the square-root taken so that
$\Im(V(\nu,\beta))>0$ if $\Im \beta>0$. (Because $c$ is small, the region contained between
$T_1\cup T_2\cup T_3$ and $\R_+$ does not contain any pole of $h_\beta$.)
Now we rewrite the integral (\ref{eq:int-division}) as follows
\begin{equation}\label{eq:cange-of-contour}
2 \int_0^\infty e^{-Nh_\beta(u)}\di u = 2\left[ \int_{T_1} + \int_{T_2} +  \int_{T_3}\right]  e^{-Nh_\beta(u)}\di u = I + E' + E.
\end{equation}
The reason for this change of contour is that in order to estimate the term $I$, we would like to perform a change of variables given by Levinson's theorem, and we would like for the obtained contour (as a result of this change) to be an interval $[0, \nu]\subset\R$.

\noindent First, we estimate the error term $E'$. We perform the  
change of variables $u = \sqrt v$ and obtain
\[
E' = 2 \int_{T_2}  e^{-Nh_\beta(u)}\di u = \int_{\widetilde{T}_2}  e^{-N\hb(v)}\frac{\di v}{\sqrt v},
\]
where $\tilde T_2$ is the push forward of 
$T_2$ by the change of variables, and has endpoints
$V(\nu, \beta), \sqrt 6 \nu$.
Now we perform another change of variables $v = V(z, \beta)$ with $V(z, \beta)$ given by Levinson's theorem, and obtain, after another contour modification,
\[
\int_{\widetilde{T}_2}  e^{-N\hb(v)}\frac{\di v}{\sqrt v} = \int_\nu^{U(\sqrt 6\nu, \beta)} e^{-N(\frac{z^2}{2} - \xib z)}\frac{V'(z, \beta)}{\sqrt{V(z,\beta)}}\di z.
\]
Since $U'(0, \beta) = \frac{1}{V'(0, \beta)}$, using the expression (\ref{en:V'0}) for $V'(0, \beta)$ we get
\[
\begin{split}
U(\sqrt 6 \nu, \beta) &= U'(0, \beta)\sqrt 6 \nu + O(\nu^2) = \frac{\beta - 1}{2\beta\xib}\sqrt 6 \nu + O(\nu^2) \\&= \frac{1}{2\xi'(1)} (1 + O(\beta - 1))\sqrt 6 \nu + O(\nu^2) = \frac{\sqrt 6 \nu}{2\xi'(1)}(1 + O(\nu)) = \frac{\sqrt 6 \nu}{2\sqrt{3/2}}(1 + O(\nu)) \\ & = \nu(1 + O(\nu)),
\end{split}
\]
where we used that $\xi'(1) = \sqrt{\frac{3}{2}}$, see the computation (\ref{eq:xi'1}) below. Therefore, for $z$ in the segment $[\nu, U(\sqrt 6 \nu, \beta)]$ we obtain
\[
\Re\left(\frac{z^2}{2} - \xib z\right) = \frac{\nu^2}{2} + O(\nu^2) \geq \frac{\nu^2}{4}.
\]
Since on this segment $|V'(z, \beta)| \neq 0$ and $c \leq |V(z, \beta)| \leq C$, we get
\begin{equation}\label{eq:E'-est}
|E'| \leq \left| \int_\nu^{U(\sqrt 6\nu, \beta)} e^{-N(\frac{z^2}{2} - \xib z)}\frac{V'(z, \beta)}{\sqrt{V(z,\beta)}}\di z \right| \leq Ce^{-CN\nu^2}.
\end{equation}

\noindent Next, we estimate the error term $E$. Note that for some small $c > 0$ 
\[
\Re\frac{1}{\beta} = \frac{1 + \epsilon}{(1 + \epsilon)^2 + R^2} = 1 - \frac{\epsilon + \epsilon^2 + R^2}{(1 + \epsilon)^2 + R^2} \geq 1 - c|\beta - 1|.
\]
Set $b = \Re(1/\beta)$.
Then, $h'_b(u) = b\, u - \tanh u$ and it vanishes on $(0, \infty)$ at a single point $u^*$ which is of order $|\beta - 1|^{1/2}$, while $h'_b(u)\to_{u\to\infty} \infty$. Hence, $h'_b(u) > 0 $ for any $u > u^*$, in particular, this holds for any $u\geq 6^{1/4}\sqrt\nu > u^*$. Note that $h_b( 6^{1/4}\sqrt\nu) \geq c\nu^2 > 0$ and this is the minimum of $h_b(u)$ on the interval $[ 6^{1/4}\sqrt\nu, C]$ for any $C >  6^{1/4}\sqrt\nu$. Since $\lim_{u\to\infty}\frac{h_b(u)}{u^2} = \tilde{c} > 1/2$, there exists $\widehat{C}$ such that $h_b(u) > u^2/2$ for any $u >\widehat{C}$. Therefore, we obtain
\begin{equation}\label{eq:E-estim}
|E| \leq \left[\int_{ 6^{1/4}\sqrt\nu}^{\widehat{C}} + \int_{\widehat{C}}^\infty\right]e^{-N h_b(u)} \di u \leq (\widehat{C} -  6^{1/4}\sqrt\nu)e^{-N h_b( 6^{1/4}\sqrt\nu)} + \frac{2}{N}e^{-N\widehat{C}}.
\end{equation}
Now we estimate the main term $I$. First, we perform  the
change of variables $u = \sqrt{v}$ and obtain
\begin{equation}\label{eq:i-changed}
I = 2 \int_{T_1}  e^{-Nh_\beta(u)}\di u = \int_{\widetilde{T}_1} e^{-N\hb(v)}\frac{\di v}{\sqrt v},
\end{equation}
where $\widetilde{T}_1$ is the push forward of $T_1$ by the change of variables.
Note that $v = 0$ is not a critical point for $\hb(v)$ for $\beta \neq 1$.
However, it is the boundary of the integration in \eqref{eq:i-changed}, 
therefore it may give a non-vanishing contribution to the value of the integral.

We perform one more change of variables $v = V(z, \beta)$ with $V(z, \beta)$ given by Levinson's theorem, and modify the contour of integration
to obtain
\[
\int_{\widetilde{T}_1} e^{-N\hb(v)}\frac{\di v}{\sqrt v} = \int_0^\nu e^{-N(\frac{z^2}{2} - \xib z)}\frac{V'(z, \beta)}{\sqrt{V(z,\beta)}}\di z. 
\]
Note that, around $z = 0$ we obtain
\begin{equation}\label{eq:V0-term}
\begin{split}
\frac{V'(z, \beta)}{\sqrt{V(z,\beta)}} = \frac{V'(0, \beta) + O(z)}{\sqrt{V(0,\beta) + V'(0, \beta)z + O(z^2)}} &= \sqrt{\frac{V'(0, \beta)}{z}}(1 + O(z)) \\&= \sqrt{\frac{2\beta\xib}{(\beta - 1)z}}(1 + O(z)),
\end{split} 
\end{equation}
where in the last equality we used that $V(0, \beta) = 0$, the condition (\ref{v-one-to-one}), and the  computation (\ref{en:V'0}) of the value $V'(0, \beta)$. In the same way we obtain around $z = \xib$
\begin{equation}\label{eq:Vxib-term}
\begin{split}
\frac{V'(z, \beta)}{\sqrt{V(z,\beta)}} &= \frac{V'(\xib, \beta) + O(z - \xib)}{\sqrt{V(\xib,\beta) + V'(\xib, \beta)(z - \xib) + O((z - \xib)^2)}} \\ &= \frac{V'(\xib, \beta)}{\sqrt{V(\xib,\beta)}} \left[1 + C_1(\beta)(z - \xib) +  O((z - \xib)^2) \right] \\ & = \frac{2\beta}{\sqrt{\beta - \beta^2 + \ub^2}} \left[1 + C_1(\beta)(z - \xib) +  O((z - \xib)^2) \right],
\end{split}
\end{equation}
where in the last equality we used the results (\ref{en:Vxib}) and (\ref{en:V'xib}) for the values $V(\xib, \beta)$ and $V'(\xib, \beta)$. 
We note that for all $\delta > 0$ the implicit constants and $C_1(\beta)$ in (\ref{eq:V0-term}) and (\ref{eq:Vxib-term}) are uniform in $\beta$ with $\delta < |\beta - 1|\leq c'$.


 Now we perform one more change of the contour of integration. We change the contour to $\Gamma_1\cup\Gamma_2\cup\Gamma_3$, where $\Gamma_i$ are the following intervals 
\begin{equation*}
\Gamma_1 = [0, -\overline\xib],\quad \Gamma_2 = [-\overline\xib, \nu + i\Im\xib], \quad \Gamma_3 = [\nu + i\Im\xib, \nu].
\end{equation*}
%
Denote for $j = 1, 2, 3$,
\[
I_j = \int_{\Gamma_j}  e^{-N(\frac{z^2}{2} - \xi(\beta)z)} \frac{V'(z, \beta)}{\sqrt{V(z, \beta)}}(1 + O(z)) \di z.  
\]
Recall that $\xib=-2h_\beta(\ub)$, see 
\eqref{en:xib}.
Our main estimate is the following.
%
\begin{lemma}\label{lem:est-on-gamma} Let $\delta < |\beta - 1| \leq c'$.
\begin{enumerate}
  \item \begin{align}
      \label{eq-ich2}
      & \left| I_1  -  \sqrt{\frac{2\pi\beta}{N(1 - \beta)}} \right|\leq \frac{C(\delta)}{N^{3/2}}.
    \end{align}
\item For any $\eta > 0$, 
\begin{align}\label{lem-eq:i2-first-case}
&\text{for $\Re\xib > \eta$}\quad
\left|I_2  - \frac{2\beta}{\sqrt {\beta - \beta^2 + \ub^2}} e^{N\frac{\xib^2}{2}}\sqrt{\frac{2\pi}{N}}\right|\leq \frac{C(\delta, \eta)}{N^{3/2}}e^{N\frac{\Re\xib^2}{2}},\\
\label{lem-eq:i2-second-case}
&\text{for $\Re\xib\geq 0$}\quad |I_2| \leq \frac{C(\delta)}{\sqrt N}e^{N\frac{\Re\xib^2}{2}},\\
\label{lem-eq:i2-third-case}
&\text{for $\Re\xib \leq 0$}\quad |I_2| \leq C(\delta)e^{-N\left(\frac{\Re\xib^2}{2} + |\xib|^2\right)},
\end{align}
\item $|I_3| \leq O(e^{-cN})$.
\end{enumerate}
\end{lemma}
Given Lemma \ref{lem:est-on-gamma}, we now complete
the proof of
Theorem
\ref{th:pacman}.
\begin{proof}[Proof of Theorem \ref{th:pacman}] The result follows from the combination of estimate (\ref{eq:E'-est}) on $E'$, the estimate (\ref{eq:E-estim}) on $E$, the definition (\ref{eq:zn-hb}) of $\zn$,  
and Lemma \ref{lem:est-on-gamma}, 
when we apply Lemma \ref{lem:est-on-gamma} as 
  follows. We consider three cases
\begin{enumerate}
\item $\Re\xib \geq \delta^{10}$, 
\item $0\leq \Re\xib \leq \delta^{10}$, 
\item $\Re\xib \leq 0$. 
\end{enumerate}
In the first case, we use the asymptotics 
(\ref{lem-eq:i2-first-case}) for $I_2$. 
In the second case, the formula (\ref{en:xib}) 
linking $\xib$ and $h_\beta(\ub)$ and the estimate 
$|\xib| \geq C_1|\beta - 1| \geq C_1\delta$ which follows from the computation of $\xi'(1)$ in  (\ref{eq:xi'1}) below imply that
\[
\Re h_\beta(\ub) = -\frac{1}{2}\Re\xib^2 \geq -\frac{1}{2}(\delta^{20} - C_2\delta^2) \geq C_3\delta^2.
\]
Therefore, the second term in the statement \ref{pacman-th:eq2} of the Theorem is subdominant. In this case we use the estimate (\ref{lem-eq:i2-second-case}) for $I_2$. In the third case, we are even further to the left of the critical curve $\Gamma$, and we use the rough estimate (\ref{lem-eq:i2-third-case}) for $I_2$.
\smallskip

\noindent In all the three cases, we use the first statement of the Lemma for $I_1$ and the third statement for $I_3$. This finishes the proof.
\end{proof}
\begin{proof}[Proof of Lemma \ref{lem:est-on-gamma}] 
We start with the estimate of $I_1$. 
Assume $\Im\xib > 0$ (the case $\Im\xib < 0$ is done in the same way). 
Define change of variables $z = -\overline\xib t,\, t\in\R$. 
Then, for $z\in[0, -\overline\xib]$ we get $t\in[0, 1]$ and
\begin{equation}
  \label{eq-ich1}
  \begin{split}
I_1 & = \sqrt{\frac{2\beta\xib}{\beta - 1}} \int_{\Gamma_1} e^{-N(\frac{z^2}{2} - \xi(\beta)z)} \frac{1}{\sqrt z}  (1 + O(z)) \di z \\&= 
\sqrt{\frac{2\beta|\xib|^2}{\beta - 1}} \int_0^1 e^{-N(\frac{(\overline\xib)^2}{2}t^2 + t|\xib|^2)}\frac{1}{\sqrt{t}}(1 + O(t))\di t\\
&= 
\sqrt{\frac{2\beta|\xib|^2}{\beta - 1}} 
\int_{-1}^1 e^{-N(\frac{(\overline\xib)^2}{2}y^4 + y^2|\xib|^2)}
(1 + O(y^2))\di y,
\end{split}
\end{equation}
where we used the change of variables $t=y^2$. Note that
the unique minimum of $g(y)=
  \frac{\Re\overline\xib^2}{2}y^4 + y^2|\xib|^2$ is at $y = 0$. 
 Indeed, if $\Re\overline\xib^2 \geq 0$, then $g(y)$ is 
 a monotone increasing function on $\R_+$ with a unique minimum at $y = 0$. 
 If $\Re\overline\xib^2 < 0$, then for any $0\leq y \leq 1$,
\[
g'(y) = 2y^3 \Re\overline\xib^2 + 2y|\xib|^2 \geq 2y
(|\xib|^2 - y^2(-\Re\overline\xib^2)) \geq 0,
\]
and the last inequality follows from $|\Re\overline\xib^2| \leq |\xib|^2$. 
We can now apply the Laplace method (for example, in the form
of \cite[Theorem 3.5.3]{AGZ}, keeping track of the error term in the proof)
to the last integral in
\eqref{eq-ich1}, and conclude with \eqref{eq-ich2}. 

\noindent Now we treat $I_2$. First, we prove the first two cases (\ref{lem-eq:i2-first-case}) and (\ref{lem-eq:i2-second-case}), where $\Re\xib\geq 0$. Using the expansion (\ref{eq:Vxib-term}) of the non-exponential term in the integral we obtain
\begin{equation}\label{eq:i2-expansion}
I_2 = \frac{2\beta}{\sqrt{\beta - \beta^2 + \ub^2}}\int_{\Gamma_2} e^{-N(\frac{z^2}{2} - \xib z)}[1 + C_1(\beta)(z - \xib) + O((z - \xib)^2)]\di z.
\end{equation}
Define the following change of variables $z = \xi(\beta) + t$, $t\in\R$. Then, $\frac{z^2}{2} - \xi(\beta)z = -\frac{\xi^2}{2} + \frac{t^2}{2}$. Note that, $\Im(-\frac{\xi^2}{2} + \frac{t^2}{2}) = \Im (-\frac{\xi^2}{2}) = \text{const}$, thus this is a minimal phase contour, and for $\Re\xib\geq 0$ it passes throughout the critical point $\xib$. Therefore, the main contribution to the integral on this contour comes from the saddle point and the rest is small. With this change of variable, we obtain
\begin{equation}\label{eq:i2}
\begin{split}
I_2 &=  \frac{2\beta}{\sqrt{\beta - \beta^2 + \ub^2}}e^{N\frac{\xib^2}{2}}\int_{-2\Re\xib}^{\nu - \Re\xib}e^{-N\frac{t^2}{2}}[1 + C_1(\beta)t + O(t^2)]\di t  \\&= \frac{2\beta}{\sqrt{\beta - \beta^2 + \ub^2}}e^{N\frac{\xib^2}{2}} [J_1 + J_2 + J_3].
\end{split}
\end{equation}
We begin with the first case (\ref{lem-eq:i2-first-case}). In this case,
the result is an immediate (elementary) application of the Laplace method,
see again \cite[Theorem 3.5.3]{AGZ}.
The correction of order $O(N^{-1})$ in (\ref{lem-eq:i2-first-case}) 
comes from the estimate on $J_3$, therefore we have finished with this case. 
Note that the implicit constant is not uniform in $\eta\to +0$. 

 To prove the estimate (\ref{lem-eq:i2-second-case}) we do the following rough bound
\[
|I_2|\leq C(\delta)e^{N\frac{\Re\xib^2}{2}}\int_{-\infty}^\infty e^{-N\frac{t^2}{2}}\di t \leq \frac{\widetilde{C}(\delta)}{\sqrt N} e^{N\frac{\Re\xib^2}{2}}.
\]
Note that $C(\delta)$ and $\widetilde{C}(\delta)$ are uniform in $\Re\xib\geq 0$. 

 Now we prove the last case (\ref{lem-eq:i2-third-case}). If $\Re\xib \leq 0$, then $\Re(-\overline\xib) \geq 0$. At the point $z = -\overline\xib$ we obtain
\[
\left[\frac{z^2}{2} - \xib z\right]\big|_{z = -\overline\xib} = \frac{\overline\xib^2}{2} + |\xib|^2 > 0.
\]
Note that $\frac{z^2}{2} - \xib z$ is a monotone increasing function on the interval $\Gamma_2$ with a minimum attained at $z = -\overline\xib$. Thus, we obtain
\[
|I_2|\leq \left| \int_{-\overline\xib}^{\nu + i\Im\xib} e^{-N(\frac{z^2}{2} - \xib z)}\max_{z\in\Gamma_2}\frac{|V'(z, \beta)|}{\sqrt{|V(z, \beta)|}}\di z\right| \leq C(\delta)e^{-N(\frac{\Re\xib^2}{2} + |\xib|^2)}.
\]

To estimate $I_3$, note that on this contour $z\in i\R$, therefore, we get for $\tilde{z} = \Im z$
\begin{equation}\label{eq:i3}
|I_3| \leq |\sqrt{V'(0, \beta)}|\int_{\nu}^{\nu + i\Im\xib}e^{-N(-\frac{\tilde{z}^2}{2} + \Im\xib\tilde{z})}\frac{\di \tilde{z}}{\sqrt{\tilde{z}}} (1 + O(\tilde{z})) \leq C e^{-cN\nu\Im\xib},
\end{equation}
where the last inequlity follows since the function $-\frac{\tilde{z}^2}{2} + \Im\xib\tilde{z}$ is  monotone decreasing for $\tilde{z} \geq \Im\xib$ with a minimum attained at $\nu + i\Im\xib$.
\end{proof}
\smallskip

\subsection{Construction of the critical curve}

\begin{proof}[Proof of Claim \ref{cl:crit-curve}]
  First, let us note the following
\begin{equation}\label{eq:hb-asym}
h_\beta(u) = \frac{u^2}{2\beta} - \log\cosh u = \frac{u^2}{2} \left(\frac{1}{\beta} - 1\right) - \left(\log\cosh u - \frac{u^2}{2} \right) = \frac{u^2}{2} \left(\frac{1}{\beta} - 1\right) + \frac{u^4}{12} + O(u^6),
\end{equation}
where the last equality holds since $\log\cosh u - \frac{u^2}{2} = - \frac{u^4}{12}  +O(u^6)$. By Claim \ref{cl:zeros-h} we get
\begin{equation}\label{eq:u-asym}
\ub^2 = 3(\beta - 1)(1 + O(\beta - 1)),
\end{equation}
therefore we obtain
\begin{equation}\label{eq:hbub-asym}
\begin{split}
h_\beta(\ub) &= \frac{\ub^2}{2\beta}(1 - \beta) + \frac{\ub^4}{12} + O(\ub^6) \\&= -\frac{3}{2\beta}(\beta - 1)^2(1 + O(\beta - 1)) + \frac{9}{12}(\beta - 1)^2(1 + O(\beta - 1)) + O((\beta - 1)^3) \\& = -\frac{3}{4}(\beta - 1)^2(1 + O(\beta - 1)).
\end{split}
\end{equation}
From the equation (\ref{en:xib}) linking $\xib$ and $h_\beta(\ub)$ we obtain that $\Re h_\beta(\ub) = 0$ if and only if $\xib\in e^{i\frac{\pi}{4}}\R\cup e^{-i\frac{\pi}{4}}\R$. Combining (\ref{en:xib}) and (\ref{eq:hbub-asym}) we conclude that 
\begin{equation}\label{eq:xi'1}
\xi'(1) = \sqrt{\frac{3}{2}},
\end{equation}
therefore, $\xib$ is one to one in $|\beta - 1| < c$ for $c > 0$ sufficiently small. Then, the curves $\gamma_\pm = \{\beta\,\,|\,\,  |\beta - 1| < c,\,\,  \xib\in e^{i\frac{\pi}{4}}\R\cup e^{-i\frac{\pi}{4}}\R\}$ are analytic. Note that, $\gamma_- = \overline\gamma_+$.

By (\ref{eq:hbub-asym}) we have $h_\beta(\ub) = -\frac{3}{4}(\beta - 1)^2 + O((\beta - 1)^3)$, therefore, for $\beta = 1 + \epsilon + iR \in\gamma_\pm$ we get
\[
0 = \Re h_\beta(\ub) = -\frac{3}{4} (\epsilon^2 - R^2) + O(\epsilon^3 + R^3),
\]
namely, $R^2 = \epsilon^2 + O(\epsilon^3)$ and we get
\[
R = \pm\epsilon (1 + O(\epsilon)).
\]
\end{proof}

\subsection{Proof of Corollary \ref{cor:zeros-of-zn}}
We consider the zeros of
\[
\psin = \frac{1}{\sqrt{1 - \beta}} + 2\sqrt{\frac{\beta}{\beta - \beta^2 + u_\beta^2}}e^{N(-h_\beta(u_\beta))}.
\] 
We will work with $\Re\beta \geq 1$ in the domain $\mathcal{D}_\delta = \{ \Re\beta \geq 1\} \cap \{\delta < |\beta - 1| < c' \}$.
First, we need the following estimate.
\begin{claim}\label{cl:est-re-hb} In the domain $\mathcal{D}_\delta$
  \begin{equation}
    \label{eq-ich3}
\frac{1}{C(\delta)}\,\, \mathrm{dist}(\beta, \Gamma) \leq |\Re h_\beta(\ub)| \leq C(\delta)\,\, \mathrm{dist}(\beta, \Gamma).
\end{equation}
\end{claim}
\begin{proof}[Proof of Claim \ref{cl:est-re-hb}] We note that $\beta\mapsto h_\beta(u_\beta)$ is Lipschitz with constant $C$ (independent of $\delta$) 
  on $\mathcal{D}_\delta$.
  This follows from the analyticity of $\xi(\beta)$ and 
  \eqref{en:xib}.

  We begin with the proof of the upper bound in \eqref{eq-ich3}.
  If $\beta'$ is the point on the critical curve $\Gamma$ closest to $\beta$, then, since $\Re h_{\beta'}(u_{\beta'}) = 0$, we get from the Lipschitz property,
\[
|\Re h_{\beta}(u_{\beta})| = |\Re h_{\beta}(u_{\beta}) - \Re h_{\beta'}(u_{\beta'})| \leq C|\beta - \beta'|
=C\mathrm{dist}(\beta, \Gamma).
\]

We turn to the proof of the lower bound in \eqref{eq-ich3}.
We have
\begin{equation}\label{eq:hb-der}
\frac{\di}{\di \beta}h_\beta(\ub) = -\frac{\ub^2}{2\beta^2} + \frac{\di \ub}{\di\beta}\left(\frac{\ub}{\beta} - \tanh\ub\right) = -\frac{\ub^2}{2\beta^2},
\end{equation}
where the last equality holds since $\ub$ is a saddle point of $h_\beta(u)$. 
By Claim \ref{cl:zeros-h} we get
on $\mathcal{D}_\delta$, 
\[
|\ub|^2 = |3(\beta - 1)(1 + O(\beta - 1))| \geq C(\delta).
\]
and therefore, on 
$\mathcal{D}_\delta$, 
\begin{equation}
  \label{eq-ich4}
  |\frac{d}{d\beta}\Re h_\beta(\ub)| \geq C'(\delta).
\end{equation}
%
Connect $\ub$ to some $\beta'\in \Gamma$ by a curve following the gradient
$\frac{d}{d\beta}\Re h_\beta(\ub)$. The length of this curve  
is bounded by a constant times the 
Euclidean distance between $\beta$ and $\beta'\in\Gamma$. Applying
\eqref{eq-ich4} then yields the lower bound, 
since $\Re h_{\beta'}(u_{\beta'})=0$.
\end{proof}

Now we observe the following
\begin{claim}\label{cl:psi-properties}\hfill
\begin{itemize}
\item The zeros of $\psin$ in $\mathcal{D}_\delta$ lie within $\frac{K(\delta)}{N}$ from $\Gamma$,
\item For any $\delta, K > 0$, there exists $C_{K, \delta}$ such that for $\beta\in\mathcal{D}_\delta$ with $\mathrm{dist}(\beta, \Gamma)\leq\frac{K(\delta)}{N}$ we have $|\Psi'_N(\beta)| \geq C_{k, \delta}N$,  $|\Psi''_N(\beta)| \leq CN^2$, where $C > 0$ does not depend on $K, \delta$. 
\end{itemize}
\end{claim}
\begin{proof}[Proof of Claim \ref{cl:psi-properties}] We start with the first statement. If $\Re h_\beta(\ub) > 0$ and the distance $\mathrm{dist}(\beta, \Gamma) > \frac{K}{N}$, then using the lower bound of Claim \ref{cl:est-re-hb}, we obtain
\[
|\psin| \geq \frac{1}{|\beta - 1|^{1/2}} - 2\left| \frac{\beta}{\beta - \beta^2 + \ub^2}\right|^{1/2}e^{-\frac{K}{C(\delta)}}.
\]
When $K = K(\delta)$ is sufficiently large, the right hand side is strictly greater than $0$.

Similarly, if $\Re h_\beta(\ub) <  0$ and $\mathrm{dist}(\beta, \Gamma) > \frac{K}{N}$, then, using again the lower bound of Claim \ref{cl:est-re-hb}, we obtain
\[
|\psin| \geq  - \frac{1}{|\beta - 1|^{1/2}} + 2\left| \frac{\beta}{\beta - \beta^2 + \ub^2}\right|^{1/2}e^{\frac{K}{C(\delta)}} > 0,
\]
for sufficiently large $K = K(\delta)$. Therefore, the zeros of $\psin$ in $\mathcal{D}_\delta$ lie in $\{\mathrm{dist}(\beta, \Gamma) \leq \frac{K}{N} \}$.
\smallskip

\noindent Now we prove the second statement. By a direct computation we get
\[
\begin{split}
\Psi'_N(\beta) &= \frac{1}{2(\beta - 1)^{3/2}} \\&+ 2\sqrt{\frac{\beta}{\beta - \beta^2 +\ub^2}}e^{-Nh_\beta(\ub)}\left[-N\left\{  \frac{\partial}{\partial u}h_\beta(u)|_{u = \ub}\frac{\partial}{\partial \beta}\ub + \frac{\partial}{\partial \beta}h_\beta(u)|_{u = \ub}\frac{\partial \beta}{\partial \beta}\right\} \right] \\& + 2\sqrt{\frac{\beta}{\beta - \beta^2 +\ub^2}}e^{-Nh_\beta(\ub)}\frac{\beta + \frac{\ub^2}{\beta} - 2\ub\frac{\partial}{\partial \beta}\ub}{\beta - \beta^2 +\ub^2}.
\end{split}
\]
Since $\ub$ is a saddle point of $h_\beta(u)$ we get $\frac{\partial}{\partial u}h_\beta(u)|_{u = \ub} = 0$ and by (\ref{eq:hb-der}) we get $\frac{\partial}{\partial \beta}h_\beta(u)|_{u = \ub} = -\frac{\ub^2}{2\beta^2}$, therefore
\[
\Psi'_N(\beta) = \frac{1}{2(\beta - 1)^{3/2}} + \sqrt{\frac{\beta}{\beta - \beta^2 +\ub^2}}e^{-Nh_\beta(\ub)} \left[N\frac{\ub^2}{2\beta^2} + \frac{\beta + \frac{\ub^2}{\beta} - 2\ub\frac{\partial}{\partial \beta}\ub}{\beta - \beta^2 +\ub^2} \right].
\]
Using the upper bound of Claim \ref{cl:est-re-hb} we obtain for sufficiently large $N$
\[
|\Psi'_N(\beta)| \geq -C_1(\delta) + e^{-NC(\delta)\frac{K(\delta)}{N}}C_2(\delta)[NC_3(\delta) - C_4(\delta)] \geq C_{K, \delta}N.
\]
The bound $|\Psi''_N(\beta)|\leq C N^2$ is obtained in the same way.
\end{proof}
Let $C_\delta > 0$. The properties of $\psin$ listed in Claim \ref{cl:psi-properties} imply that the distance between any two zeros of $\psin$ in $\mathcal{D}_\delta$ at least $\geq \frac{C_\delta}{N}$. Indeed, let $\beta_0$ be a zero of $\psin$. Then,
\[
\psin = \Psi_N(\beta_0) + \Psi'_N(\beta_0)(\beta - \beta_0) + O(N^2)(\beta - \beta_0)^2.
\]
Since $|\Psi'_N(\beta_0)| \geq C_{K, \delta}N$, we obtain
\[
|\psin| \geq  C_{K, \delta}N|\beta - \beta_0| - CN^2|\beta - \beta_0|^2,
\]
for every $|\beta - \beta_0|\leq cN^{-1}$. In particular, $|\psin| = 0$ implies $|\beta - \beta_0| \geq \frac{C_{\delta}}{N}$.


Now we look at the discs of radius $C_\delta N^{-2}$ around each zero of $\psin$ near $\Gamma$ and we claim that there is exactly one zero of $\zn$ in each disc. By an additional application of Rouch\'e's theorem it is sufficient to show that for sufficiently large $C_\delta$ we have on the boundary of each disc
\begin{equation}\label{eq:zn-psin}
|\zn - \psin|\leq\frac{|\psin|}{2}.
\end{equation}
The estimate (\ref{eq:zn-psin}) follows since $|\zn - \psin|\leq O(N^{-1})$ uniformly in $\{\delta < |\beta - 1| \leq c',\,\,\, \mathrm{dist}(\beta, \Gamma) \leq \frac{2K(\delta)}{N}\}$ and on the boundary of each disc of radius $C_\delta N^{-2}$ we get $|\psin| \geq C_{K,\delta}N \frac{C_\delta}{ N^{2}} = \frac{C_0}{N}$, where $C_0$ may be made arbitrarily large by adjusting $C_\delta$. Therefore, there is exactly one zero of $\zn$ in each of these discs.

\smallskip
\noindent To show that there are no additional zeros of $\zn$ in $\delta < |\beta - 1| < c'$, first we observe that, by Theorem \ref{th:pacman}, the zeros of $\zn$ in $\delta < |\beta - 1| < c'$ lie in $\Re\beta \geq 1$. 
\smallskip

Consider the domain $\widetilde{\mathcal{D}}_\delta = \{ \widetilde\delta < |\beta - 1| \leq \widetilde{c},\,\, \Re\beta \geq 1 \}$, where $\widetilde\delta $ and $\widetilde{c}$ are such that
\[
\delta - \frac{C_\delta}{N} \leq \widetilde\delta \leq \delta,\,\,\,\, c' \leq \widetilde{c} \leq c' + \frac{C_\delta}{N},
\]
and the distance $\mathrm{dist}(\beta, \partial\widetilde{\mathcal{D}}_\delta) \geq \frac{C_2}{N}$ for any zero $\beta$ of $\psin$. We will
check that the inequality (\ref{eq:zn-psin}) holds on the boundary $\partial\widetilde{\mathcal{D}}_\delta$, then by Rouch\'e 's 
theorem the zeros of $\zn$ in $\widetilde{\mathcal{D}}_\delta$ are exactly those constructed in the first part of the proof.

We divide the boundary $\partial\widetilde{\mathcal{D}}_\delta$ of the domain $\widetilde{\mathcal{D}}_\delta$ as follows: $\partial\widetilde{\mathcal{D}}_\delta = A\cup B$, where $A = \partial\widetilde{\mathcal{D}}_\delta \cap \{ \mathrm{dist}(\beta, \Gamma)\leq \frac{K(\delta)}{N}\}$ and $B = \partial\widetilde{\mathcal{D}}_\delta\setminus A$. For sufficiently large $K(\delta)$ the inequality (\ref{eq:zn-psin}) is valid on $B$ by Theorem \ref{th:pacman}. 
We now
 show that (\ref{eq:zn-psin}) also holds on $A$. The set $\{ \widetilde\delta < |\beta - 1| < \widetilde{c},\,\,\, \mathrm{dist}(\beta, \Gamma) \leq \frac{K(\delta)}{N}\}$ contains $A$, therefore we have uniformly in $A$
\[
|\zn - \psin|\leq O(N^{-1}).
\]
Also, as before, we have $|\psin| \geq C_{K, \delta}N\frac{C_\delta}{N} = C_0$ on $A$. Thus, the inequality (\ref{eq:zn-psin}) holds on $A$, and we conclude the proof.
\qed

\subsection{Proof of Corollary \ref{cor:emp-mes}} 
Define
\[
\widetilde{\mu}_N = \frac{1}{N}\sum_{\beta:\;|\beta - 1|\leq c',\,\ \Re\beta \geq 1,\,\, \psin = 0}\delta_\beta.
\]
We will show that
\begin{eqnarray}
  \label{en:tilde-mun-conv} &&\widetilde\mu_N \to_{N\to\infty} \mu,\\
\label{en:mu-conv} &&\mu\{ |\beta - 1| < \delta\} \to_{\delta \searrow 0}\,\, 0,\\ 
\label{muN-conv} &&\limsup_{N\to\infty}\mu_N\{ |\beta - 1| < \delta\} 
\to_{\delta \searrow 0}\,\, 0 .
\end{eqnarray}
Choose $\delta > 0$. Then, by \eqref{en:tilde-mun-conv} 
and Corollary \ref{cor:zeros-of-zn} we obtain
\[
\mu_N\restriction_{\delta \leq |\beta - 1| \leq c'}\,\, \to \mu\restriction_{\delta \leq |\beta - 1| \leq c'}.
\]
Using \eqref{en:mu-conv} and \eqref{muN-conv}, 
and letting $N\to \infty$ and then $\delta\to 0$ we obtain $\mu_N \to \mu$.
It remains to show \eqref{en:tilde-mun-conv}, \eqref{en:mu-conv} and 
\eqref{muN-conv}. 

Toward this end, note that since 
$\xi'(1) = \sqrt{{3}/{2}} \neq 0$, see 
(\ref{eq:xi'1}), it follows that
$\xib$ is one-to-one in a neighborhood of $\beta=1$, and in fact
it maps a neighborhood of $\beta = 1$ biconformally onto a neighborhood of 
$0$. In particular, with $I$ denoting the
line segment $[0, c''e^{\frac{i\pi}{4}})$ with $c''>0$ small, 
we have by 
Claim \ref{cl:crit-curve} that $\xi^{-1}(I)$ 
is a segment of $\Gamma\cap \{\Im \beta \geq 0\}$
containing $\beta=1$.
Therefore, by \eqref{en:xib},  
$h_\beta(\ub)$ maps $\xi^{-1}(I)$ bijectively onto $(0, i c)$ for some $c>0$.
A similar argument applies with
$I_-=[0, -c''e^{\frac{i\pi}{4}})$ replacing $I$ and
 $\Gamma\cap \{\Im \beta \leq 0\}$ replacing  
 $\Gamma\cap \{\Im \beta \geq 0\}$

Let $\beta_k\in\Gamma$, $k\in \Z$,
be such that $h_{\beta_k}(u_{\beta_k}) = \frac{2\pi i k}{N}$ 
is smaller in absolute value than $c$. 
It follows from the above considerations that
\[
\frac{1}{N} \sum_{|\beta_k - 1|\leq c'}\delta_{\beta_k} \underset{N\to\infty}\longrightarrow \mu,
\]
since the left-hand side and the right-hand side assign the same value to each half-open curved segment of $\Gamma$ connecting two points $\beta_l$ and $\beta_{l'}$; this value is $\frac{l - l'}{N}$.

Next, let $\widetilde{\beta}_k$ be such that $\Re\widetilde{\beta}_k \geq 1$ and
\[
h_{\widetilde{\beta}_k}(u_{\widetilde{\beta}_k}) = -\frac{1}{N}\log\left[-\frac{1}{2}\sqrt{\frac{\widetilde{\beta}_k - \widetilde{\beta}_k^2 + u_{\widetilde{\beta}_k}^2}{\widetilde{\beta}_k - \widetilde{\beta}_k^2}}\right] + \frac{2\pi i k}{N}.
\]
Then, using the relation (\ref{en:xib}) and that $\xib$ is one to one,
we get
\[
\widetilde\mu_N = \frac{1}{N}\sum_{|\widetilde{\beta}_k - 1|\leq c'}\delta_{\widetilde{\beta}_k},\,\,\,\, \sup_{k\leq N}|\widetilde{\beta}_k - \beta_k| = o(1).
\]
Hence, $\widetilde\mu_N\to\mu$, and this finishes the proof of \eqref{en:tilde-mun-conv}.

Next,  since $\xi(\cdot)$ is Lipschitz in a neighborhood of $\beta=1$, we
get from \eqref{en:xib} that 
\begin{equation}\label{eq:hb-ub-est}
|h_\beta(\ub)| \leq C|\beta - 1|^2.
\end{equation}
%
The relation \eqref{en:mu-conv} follows since
for $b(\delta)$ which is the intersection of the critical curve $\Gamma$ with $|\beta - 1| = \delta$ we obtain
\[
\mu\{ |\beta - 1| < \delta \} \leq 2|h_{b(\delta)} (u_{b(\delta)})| \leq 2C |\beta - 1|^2 \leq C\delta^2,
\]
where the two last inequalities follow from the estimate
\eqref{eq:hb-ub-est}
and from the fact that $|\beta - 1| = \delta$.

To prove the relation \eqref{muN-conv}, denote by $n_N(\delta)$ the number of zeros of $\zn$ in $|\beta - 1| \leq \delta$. Then, by Jensen's formula we obtain 
\[
n_N(\delta) = \#\{ \beta\colon\,\, |\beta - 1| \leq \delta,\,\,\, \zn = 0\} \leq \frac{1}{\log\frac{2\delta}{\delta}}\log\frac{\max_{|\beta - 1| = 2\delta}|\zn|}{|Z_N(1)|}.
\]
From the case for real $\beta$ the denominator is bounded from below by $N^{-C}$, for some $C > 0$, and we need to bound the numerator from above. We get
\[
\frac{1}{\log2}\log\frac{\max_{|\beta - 1| = 2\delta}|\zn|}{|Z_N(1)|} \leq C[N\max_{|\beta - 1| = 2\delta}|h_\beta(\ub)| + A_\delta + \log N] \leq C' [N\delta^2 + A_\delta + \log N],
\]
where the first inequality follows from Theorem \ref{th:pacman}, and the last one follows from the estimate (\ref{eq:hb-ub-est}) and since $|\beta - 1| = 2\delta$. Thus, using the last estimate, we obtain
\[
\mu_N\{ |\beta - 1|\leq \delta\} \leq C'\delta^2 + \frac{C'A_{\delta}}{N} + \frac{C'\log N}{N}.
\]
Letting first $N\to\infty$ and then $\delta\searrow 0$ we obtain \eqref{muN-conv} and thus conclude the proof.
\qed

\section{Conjecture: the critical curve}
\label{sec-open}
We conjecture that there exists a constant $b\in (0,\infty)$, a curve
\[
\{1 + \epsilon_0(R) + iR\}_{-\infty < R < \infty}
\]
such that
\[
\epsilon_0(0) = 0,\, \quad \epsilon_0(R) > 0\,\, \text{for}\,\,  R\neq 0,\, \quad \lim_{R\to+\infty}R^2\epsilon_0(R) =b,
\]
and an auxiliary function $\delta_0(R)\geq 0$ with equality only at $0$,
 so that the following holds:
\begin{equation}\label{eq:conj}
\lim_{N\to\infty}\frac{1}{N}\log|Z_{1 + \epsilon + iR, N}| \,\,
\begin{cases} 
= 0, & 0\leq\epsilon\leq \epsilon_0(R),\\
> 0, & \epsilon_0(R) < \epsilon \leq \epsilon_0(R) + \delta_0(R). 
\end{cases}
\end{equation}
Moreover, we conjecture that the curve is described by one branch of the saddle point equation, as follows.

For $f_\beta(u)$ from Proposition \ref{prop:int-rep-z}, consider the saddle point equation
\begin{equation}\label{eq:saddle-p}
f'_\beta(u) = \beta u - \beta\tanh(\beta u) = 0.
\end{equation}
It defines a multivalued function $u(\beta)$. We claim that there exists a branch $u^*(\beta)$ in $0 < \Re \beta\leq C \approx 1.3$, such that $u^*(1) = 0$.  Indeed, the equation (\ref{eq:saddle-p}) is equivalent to 
\begin{equation}\label{eq:ustar}
\beta = \frac{1}{2u}\left[\log\frac{1 + u}{1 - u} + 2\pi i k\right],\, k\in\Z,
\end{equation}
where we take the principal branch of the logarithm. For $k = 0$ the equation (\ref{eq:ustar}) defines a bijection between the first and the fourth quadrants in the $u$-plane and the domains depicted in Figure \ref{fig:u} (left). For $k = 1$ the function from the right hand side of the equation (\ref{eq:ustar}) maps the first and the fourth quadrants onto the domains in Figure \ref{fig:u} (right). 
\begin{figure}[h!]
\begin{center}
\includegraphics[width=8cm, height=5cm]{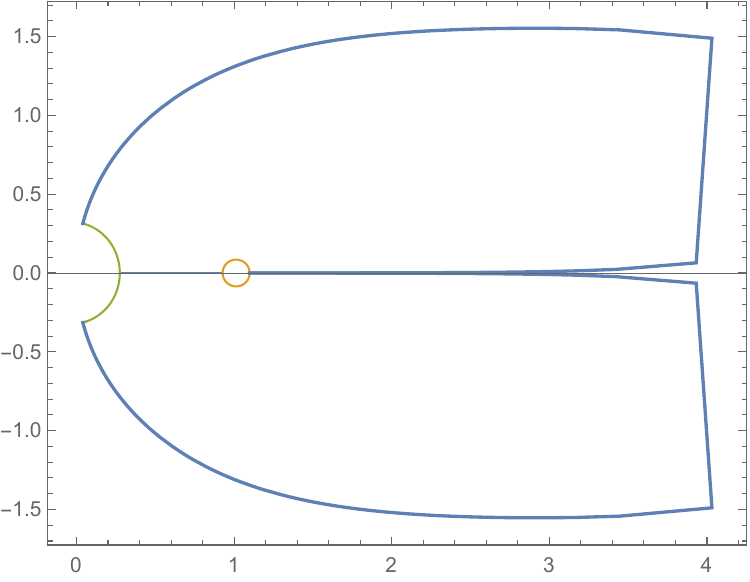}
\includegraphics[width=8cm, height=5cm]{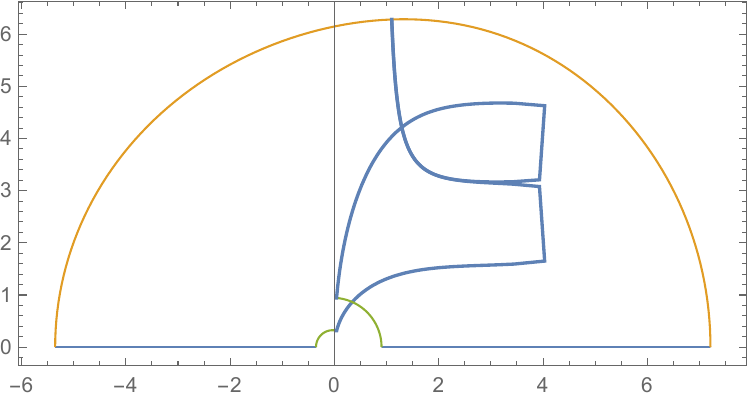}
\end{center}
\caption{The images of the I-st and IV-th quadrants in the $u$-plane under (\ref{eq:ustar}), with $k=0$ (left) and $k=1$ (right). The vertical lines in both plots and the large semi-circle at the right lie at infinity.}
\label{fig:u}
\end{figure}

\noindent Consequently, one can define a branch $u^*(\beta)$ in $\{0 < \Re\beta \leq C,\,\,\, \Im\beta > 0\}$, where $C\approx 1.3$ is the real part of the intersection point between the two curves on Figure \ref{fig:u} (right), which corresponds to $k = 0$ in the intersection with the domain in Figure \ref{fig:u} (left) and to $k = 1$ outside it. Similarly, we define $u^*(\beta)$ for $\Im\beta < 0$.
\begin{conj} The relation $\mathrm{(}$\ref{eq:conj}$\mathrm{)}$ holds with $\epsilon_0(R)$ defined by the equation
\[
\Re f_{1 + \epsilon_0(R) + iR}(u^*(1 + \epsilon_0(R) + iR)) = 0.
\]
\end{conj}


\begin{thebibliography}{10}
    \bibitem{AGZ} G.~W.~Anderson, A.~Guionnet, O. Zeitouni,
    \textit{An introduction to Random Matrices},
    Cambridge University Press, Cambridge (2010).
  \bibitem{DeZe} Dembo, A. and Zeitouni, O. 
    \textit{Large Deviations Techniques and Applications}, 2nd Ed.
    Springer, New York (1998).
  \bibitem{EN} Ellis, R. S. and Newman, C. M. \textit{Limit theorems
    for sums of dependent random variables occuring in statistical mechanics},
    Prob. th. rel. Fields {\bf 44} (1978), pp. 117--139.
\bibitem{Fisher} Fisher, M.~E. \textit{The nature of critical points}, Lecture Notes in Theoretical Physics vol 7c, Boulder: University of Colorado Press (1965), pp. 1--159
\bibitem{GPS} Glasser, M. L., Privman, V. and Schulman, L. S.,
  \textit{Complex temperature plane zeros in the mean-field approximation},
  J. Stat. Physics {\bf 45} (1986), pp. 451--457.
\bibitem{KBHK} Krasnytska, M., Berche B., Holovatch Yu. and Kenna R.,   \textit{Partition function zeros for the Ising model on complete graphs and on annealed scale-free networks}, Journal of Physics A: Mathematical and Theoretical {\bf 49.13} (2016): 135001.

\bibitem{Maillet+}
    Kitanine, N., Maillet, J.M., Slavnov, N. A. and Terras, V., 
    \textit{Large distance asymptotic behavior of the emptiness formation 
    probability of the $XXZ$ spin-$\frac12$ Heisenberg chain},
    J. Phys. A {\bf 35} (2002), L735--10502.

  \bibitem{Lev} N.~Levinson, \textit{Transformation of an analytic function of several variables to a canonical form}, Duke Mathematical Journal {\bf 28} 
   (1961), pp. 345--353
  
     \bibitem{Martin} J.~Martin, \textit{Integrals with a large parameter and several nearly coincident saddle points; the continuation of uniformly asymptotic expansions}, Mathematical Proceedings of the Cambridge Philosophical Society.
       {\bf 76}, (1974).
  \bibitem{ML} Martin-L{\"{o}}f, A. \textit{A Laplace approximation for sums of independent random variables}, Z. Wahr. ver. Geb. {\bf 59} (1982), pp. 101--115.
  \bibitem{LY} Yang, C. N. and Lee, T. D., 
    \textit{Statistical theory of equations of state and phase transitions, I.
    Theory of condensation}, Phys. Rev. {\bf 87} (3) (1952), 404--409.
  \end{thebibliography}
\end{document}